\documentclass[11pt,a4paper]{scrartcl}       
\usepackage[utf8]{inputenc}
\usepackage[english]{babel}
\usepackage{amssymb, amsmath, amsfonts, amsthm}
\usepackage{enumitem}
\usepackage{colonequals}
\usepackage{empheq,fancybox}
\usepackage{esint}
\usepackage{authblk}
\usepackage{tikz}
\usepackage{float}
\usepackage{hyperref}

\usetikzlibrary {intersections}
\numberwithin{equation}{section}

 \newcommand{\hm}[1]{\leavevmode{\marginpar{\tiny%
 $ \hbox to 0mm{\hspace*{-0.5mm} $ \leftarrow $ \hss}%
 \vcenter{\vrule depth 0.1mm height 0.1mm width \the\marginparwidth}%
 \hbox to
 0mm{\hss $ \rightarrow $ \hspace*{-0.5mm}} $ \\\relax\raggedright #1}}}
\newcommand{\euler}{\mathrm{e}} 
\newcommand{\drm}{\mathrm{d}}
\newcommand{\dvol}{\mathrm{dvol}}
\newcommand{\RR}{\mathbb{R}}

\newcommand{\cC}{\mathcal{C}}

\renewcommand{\epsilon}{\varepsilon}
\renewcommand{\phi}{\varphi}
\DeclareMathOperator{\supp}{\mathop{supp}}
\DeclareMathOperator{\diam}{\mathop{diam}}
\DeclareMathOperator{\dist}{\mathop{dist}}
\DeclareMathOperator{\Tr}{\mathop{Tr}}

\DeclareMathOperator{\Ric}{\mathop{Ric}}
\DeclareMathOperator{\Vol}{\mathop{Vol}}

\DeclareMathOperator{\Sec}{Sec}

\DeclareMathOperator{\Hess}{Hess}
\DeclareMathOperator{\II}{II}

\newcommand{\bd}{\partial}
\newcommand{\e}{\mathrm e}

\newtheorem{thm}{Theorem}[section]
\newtheorem{lem}[thm]{Lemma}
\newtheorem{prop}[thm]{Proposition}
\newtheorem{cor}[thm]{Corollary}
\theoremstyle{definition}
\newtheorem{definition}[thm]{Definition}
\newtheorem{remark}[thm]{Remark}
\newtheorem{example}[thm]{Example}
\begin{document}
\title{Quantitative Sobolev extensions and the Neumann heat kernel for integral Ricci curvature conditions}
\author{Olaf Post\footnote{Universit\"at Trier, olaf.post@uni-trier.de}~~~~~~~~Xavier Ramos Oliv\'e\footnote{Worchester Polytechnic Institute, xramosolive@wpi.edu}~~~~~~~~Christian Rose\footnote{Max Planck Institute for Mathematics in the Sciences, Leipzig, crose@mis.mpg.de}}

\maketitle
\begin{abstract}We prove the existence of Sobolev extension operators for certain uniform classes of domains in a Riemannian manifold with an explicit uniform bound on the norm depending only on the geometry near their boundaries. We use this quantitative estimate to obtain uniform Neumann heat kernel upper bounds and gradient estimates for positive solutions of the Neumann heat equation assuming integral Ricci curvature conditions and geometric conditions on the boundary. Those estimates also imply quantitative lower bounds on the first Neumann eigenvalue of the considered domains.
\end{abstract}
\section{Introduction}
The first aim of this article is to prove the existence of Sobolev extension operators for domains with smooth boundary in a Riemannian manifold whose norms depend only on the geometry near the boundary.  To the best of our knowledge, we give the first explicit construction, giving a quantitative bound on the norm, that depends only on sectional and principal curvature assumptions. Such an extension operator provides a tool for geometric applications, especially when working on \emph{classes} of manifolds fulfilling certain geometric bounds. Our second aim is to use these extension operators to derive quantitative upper bounds for the Neumann heat kernel, gradient estimates for positive solutions of the Neumann heat equation, and lower bounds on the first Neumann eigenvalue under $L^p$-Ricci curvature conditions for sufficiently regular compact domains.

Let $M=(M^n,g)$ be a complete Riemannian manifold of dimension $n\geq 2$ with possibly non-empty boundary $\partial M$. Fix an open subset $\Omega\subset M$ such that $\overline \Omega \ne M$ is a smooth manifold with boundary.  A linear and bounded operator $E_\Omega\colon H^1(\Omega)\to H^1(M)$ is called \emph{extension operator} for $\Omega$, if it is a bounded right-inverse for the restriction operator of $\Omega$, i.e., if $E_\Omega$ satisfies
\begin{align}\label{defextension}
E_\Omega u{\restriction}_\Omega = u,\quad u\in H^1(\Omega),
\end{align}
and has bounded operator norm. 
Such extension operators have a long history, starting with the seminal work by Whitney~\cite{Whitney-34} %\lookO{I think this is the paper with the ``real'' extension operator}%,Whitney-34a,Whitney-34b} Olaf: I think this reference is enough, the others are variants.
 for extension operators of class $C^k$. It is well known, cf., e.g., Stein's monography~\cite[Thm.~5, Sec.~VI.3.1]{Stein-70}, %\lookO{and~ BoossWojciechowski?%\cite{boosswoj-12} 
%--- Sollen wir Booss-etc wirklich zitieren? In Def.~11.7~((e) gibt es nur eine Definition eines Fortsetzungsoperator (Kodimension $0$), ohne Beweis, und mit glattem Rand.}
 that for a domain $\Omega \subset \RR^n$ with Lipschitz boundary, such a Sobolev extension operator $E_\Omega$ exists. Its norm depends implicitly on the Lipschitz constants of the charts as well as other properties of the atlas of $\bd\Omega$. Therefore, this construction does not imply the existence of extension operators for classes of subsets of manifolds whose norms depend only on curvature restrictions.
%The proof in~\cite{Stein-70} works for boundaries $\bd \Omega$ having a cover $(U_i)_i$ such that each point $x \in \bd \Omega$ is contained in at most $N$ sets $U_i$ and there is $R>0$ such that for each $x \in \bd \Omega$ there exists $i$ with $B(x,R) \subset U_i$, and in each set $U_i$, the set $\bd \Omega \cap U_i$ is graph of a Lipschitz function with Lipschitz constant not exceeding $L$.  
%In particular, the norm of the extension operator depends only on $N$, $R$ and $L$, but the dependence is not given explicitly. 
Sobolev extension operators especially for finite sets are constructed in~\cite{FeffermanIsraelLuli-14}, see also the references therein. The name ``extension operator'' is also used in a slightly different context, namely as a bounded right inverse of a Sobolev trace operator, i.e., of a restriction of functions to proper submanifolds, see, e.g.,~\cite{GluckZhu-19, GrosseSchneider-13} and the references therein. 

However, to the best of our knowledge, apart from Stein's result for subsets in $\RR^n$ mentioned above, existing constructions of extension operators for $\Omega\subset M$ in the sense of~\eqref{defextension} do not focus on quantitative bounds on $\Vert E_\Omega\Vert$. Our aim is to construct extension operators whose operator norms depend only on geometric quantities such as bounds on the second fundamental form and the sectional curvature in a suitable tubular neighborhood of $\partial\Omega$,. Therefore, we introduce the following class of subsets of a manifold.
% Our first main result is formulated for the sufficiently regular subsets of a Riemannian manifold introduced below.
\begin{definition}\label{defrhk}
Let $M$ be a Riemannian manifold of dimension $n\geq 2$, $r>0$ and $H,K\geq 0$. An open subset $\Omega\subset M$  is called \emph{$(r,H,K)$-regular} if
  \begin{enumerate}[label=(\roman*)]
  \item $\overline \Omega\neq M$ is a connected smooth manifold with (smooth) boundary $\partial\Omega$,
\item \label{exterior} the \emph{exterior rolling $r$-ball condition}: for any $x\in\partial\Omega$ there is a point $p \in M \setminus \Omega$ such that $B(p,r) \subset M \setminus \Omega$ and $\overline{B(p,r)}\cap\partial\Omega=\{x\}$;
\item \label{interior} the \emph{interior rolling $r$-ball condition}: for any $x\in\partial\Omega$ there is a point $p \in \Omega$ such that $B(p,r) \subset \Omega$ and $\overline{B(p,r)}\cap\partial\Omega=\{x\}$;
%\item for any $x\in\partial\Omega$, there exists a cut-off function $\varphi_x$ such that $\supp\varphi_x\subset B(x,r)$, $\varphi=1$ on $B(x,r/2)$, and $\vert \nabla\varphi_x\vert\leq G /r$;
  % \item \label{def:rolling} 
  %   $\partial\Omega$ satisfies the interior and exterior rolling $r$-ball condition, cf.~Theorem~\ref{main1}~\ref{exterior},\ref{interior},
\item \label{main:secondf}
  the second fundamental form $\II$ is bounded from below by $-H$ w.r.t.\ the inward and the outward pointing normals of $\partial\Omega$,
%  \item $\partial\Omega$ satisfies the exterior rolling $r$-ball condition, cf.~Theorem~\ref{main1}~\ref{exterior},
\item \label{main:tubular}
  the sectional curvature satisfies $-K\leq \Sec\leq K$ on the tubular neighborhood $T(\partial\Omega,r)$ of $\partial\Omega$.
\end{enumerate}
\end{definition}
Our first main result is now as follows:
\begin{thm}\label{main1} Fix $K,H\geq 0$, and a complete Riemannian manifold $M$ of dimension $n\geq 2$. There exists an explicitly computable $r_0=r_0(K,H)>0$ such that for any $r\in(0,r_0]$, there exists an explicitly computable constant $C(r,K,H)>0$ and 
 % $\Omega\subset M$ with $\overline \Omega \ne M$ being a smooth manifold with boundary $\partial\Omega$ satisfying
% \begin{enumerate}[label=(\roman*)]
% \item \label{exterior} the \emph{exterior rolling $r$-ball condition}: for any $x\in\partial\Omega$ there is a point $p \in M \setminus \Omega$ such that $B(p,r) \subset M \setminus \Omega$ and $\overline{B(p,r)}\cap\partial\Omega=\{x\}$;
% \item \label{interior} the \emph{interior rolling $r$-ball condition}: for any $x\in\partial\Omega$ there is a point $p \in \Omega$ such that $B(p,r) \subset \Omega$ and $\overline{B(p,r)}\cap\partial\Omega=\{x\}$;
% %\item for any $x\in\partial\Omega$, there exists a cut-off function $\varphi_x$ such that $\supp\varphi_x\subset B(x,r)$, $\varphi=1$ on $B(x,r/2)$, and $\vert \nabla\varphi_x\vert\leq G /r$;
% \item\label{main:secondf}the second fundamental forms $\II$ of $\partial\Omega$ w.r.t.\ the inward and outward pointing normals satisfy $\II\geq -H$;
% \item\label{main:tubular}In the tubular neighborhood $T(\partial\Omega,r)$ of $\partial\Omega$, the sectional curvature $\Sec$ satisfies $-K\leq \Sec\leq K$,
% \end{enumerate}
an extension operator $E_\Omega\colon H^1(\Omega)\to H^1(M)$ satisfying
\[
\Vert E_\Omega\Vert\leq C(K,H,r)
\]
for any open $(r,H,K)$-regular subset $\Omega$.
\end{thm}
\begin{remark}
\indent
\begin{enumerate}[label=(\roman*)]
\item
Although we assume completeness of $M$ in the above theorem, the proof also applies for incomplete $M$ and $\Omega\subset M$ in any connected component $M'$ satisfying the conditions of Theorem~\ref{main1} such that $\overline\Omega\neq M'$ and such that $\cC^1(\overline \Omega)$ is dense in $H^1(\Omega)$.
\item We do not expect that our constant $C(K,H,r)$ is sharp.  If $\Omega=B(0,R_0)$ is an open ball in $\RR^n$, then we can choose $r_0=r=R_0/2$ and $H=1/R_0$ %$G=8/R_0$ 
(see Examples~\ref{ex:sphere} and~\ref{ex:sphere2}).  A simple scaling argument shows that $\Vert E_\Omega \Vert$ is bounded from above by the norm of the extension operator on the unscaled ball $B(0,1)$ times $R_0^{-1}=(2r)^{-1}$.
\item The upper bound on the sectional curvature and lower bound on the second fundamental form ensure that the exponential map of $\partial\Omega$ is well-defined on $T(\partial\Omega,r)$ for $r\leq r_0$. The distance $r_0$ is in fact bounded by the minimal focal distance of all points in $\partial\Omega$.
\item Note that a lower bound on the second fundamental form w.r.t.~a fixed direction is equivalent to an upper bound of the second fundamental form w.r.t.~the opposite direction.
\item The upper and lower bounds for the sectional and principal curvatures on $T(\partial\Omega,r)$ and $\partial\Omega$ can be replaced by two different upper and lower bounds on $T(\partial\Omega,r)\cap\Omega$ and $T(\partial\Omega,r)\setminus\Omega$. Moreover, $\Sec$ actually need not be a priori bounded from below in $T(\partial\Omega,r)\setminus\Omega$ if one assumes additionally the existence of cut-off functions.
\item The interior rolling $r$-ball condition ensures that $\Omega$ is "sufficiently thick": there always fits a ball of a controllable size into the interior, and geodesics emanating from different points stay unique up to $r$. 
\item The exterior rolling $r$-ball condition, cf.~Figure~\ref{Figure2}, is indeed necessary: the proof of Theorem~\ref{main1} relies on a particular parametrization, the so-called \emph{Fermi-coordinates}, of the tubular neighborhood. The exterior rolling $r$-ball condition prevents the tubular neighborhood from self-overlapping, ensuring that the parametrization and the extension operator are well defined. For a problematic case, see Figure~\ref{Figure1}.

\end{enumerate}
\end{remark}
%
%
%
%
%
%
%

    %  \centering
%\lookO{In Figure~1, I changed $R/2$ to $R$, as this refers to the $R$-ball, but is everything now consistent ($R$ versus $R/2$)?\textcolor{red}{Olaf is right, I was obviously confused. Must change $R/2$ in the proof back to $R$ and replace $R$ by $r$, otherwise inconsistent with Ricci $R$.}}
\begin{figure}[ht]
  \begin{minipage}[b]{0.45\linewidth}
    \centering
    \begin{tikzpicture}\label{Figure1}[scale=1.5]%we can scale the
      % figure with this
      % command
      % If we want to draw a 1-dimensional "hypersurface" instead
      % \draw (1.73cm, 1cm) arc [start angle =30, end angle =360,
      % radius = 2cm];
      
      % Drawing the domain \Omega as a 0.2cm-tubular nbhd of the
      % circle of radius 2cm centered at (0,0)
      \draw (1.91cm,1.1cm) arc [start angle=30,end angle=360,radius=2.2cm];
      \draw (1.56cm,0.9cm) arc [start angle=30,end angle=360,radius=1.8cm];
      \draw (2.2cm,0cm) arc [start angle =0, end angle =180, radius =0.2cm];
      \draw (1.56cm,0.9cm) arc [start angle =210, end angle =390,
      radius =0.2cm];
      
      \fill[black!10!white] (1.56cm,0.9cm) arc [start angle =210, end
      angle =390, radius =0.2cm] arc[start angle=30,end
      angle=360,radius=2.2cm] (2.2cm,0cm) arc [start angle =0, end
      angle =180, radius =0.2cm] arc[start angle=360,end
      angle=30,radius=1.8cm];
 
      % Drawing the 0.6cm-tubular neighborhood of \Omega as a 0.8cm-tubular nbhd of the same circle
      \draw (1.04cm,0.6cm) arc [start angle=30,end angle=360,radius=1.2cm];
      \draw (2.43cm,1.4cm) arc [start angle=30,end angle=360, radius = 2.8cm];
      \draw (2.8cm,0cm) arc [start angle =0, end angle =180, radius =0.8cm];
      \draw (1.04cm,0.6cm) arc [start angle =210, end angle =390, radius =0.8cm];
      
      % Define the boundaries that intersect, to find the points of intersection x and y.
      \path[name path=upcircle bdy1] (2cm,0.8cm) arc [start angle =90, end angle =180, radius =0.8cm];
      \path[name path=upcircle bdy2] (2.8cm,0cm) arc [start angle =0, end angle =90, radius =0.8cm];
      \path[name path=downcircle bdy1] (1.04cm,0.6cm) arc [start angle =210, end angle =300, radius =0.8cm];
      \path[name path=downcircle bdy2] (2.43cm,1.4cm) arc [start angle =30, end angle =-60, radius =0.8cm];
      % Filling in the area where the extension operator is not well defined
      \fill[black!20!white] [name intersections={of=upcircle bdy1 and downcircle bdy1, by=x}] [name intersections={of=upcircle bdy2 and downcircle bdy2, by=y}] (x) arc [start angle = 154.6806274, end angle = 55.3193726, radius =0.8cm] arc [start angle = -25.3193726  , end angle = -124.6806274, radius =0.8cm];
      % Labels
      \draw  (-1cm , 1.73cm) node {$\Omega$};
      % \draw (3cm,3cm) node {$T(\partial \Omega, R/2)\cap(M\setminus\Omega)$};
      \draw[<->, black, very thick] (0cm,1.8 cm) -- (0cm, 1.2cm);
      \draw[black] (0.4cm, 1.4cm) node {$r$};
    \end{tikzpicture}
    % \end{center}
    \caption{A problematic neighborhood not fulfilling the exterior
      rolling $r$-ball condition.}
  \end{minipage}
  \hspace{.5cm}
  \begin{minipage}[b]{0.45\linewidth}
    \centering
    \begin{tikzpicture}[scale =1.7, yshift=5pt]\label{Figure2}%we can scale the figure with this command
      % Boundary of \Omega
      \draw[thick] (-1,0) .. controls (-0.8,0.555) and (-0.555,1) .. (0,1).. controls (0.555,1) and (0.8,0.555) .. (1,0) .. controls (1.2,-0.555) and (1.445,-1) .. (2,-1).. controls (2.555,-1) and (2.8,-0.555) .. (3,0);
      % Point x in boundary of \Omega
      \filldraw[black] (1.22,-0.5) circle (1pt);
      % Ball centered at p touching the boundary at x only
      % \draw (1.22,-0.5) arc[start angle =225, end angle =585, radius=0.5];
      \draw (1.22,-0.5) arc[start angle =210, end angle =570, radius=0.5];
      \filldraw (1.65,-0.25) circle (1pt);
      % Labels
      \draw  (-1.2 , -0.2) node {$\partial\Omega$};
      \draw[black]  (1.1 , -0.6) node {$x$};
      \draw  (1.53, -0.35) node {$p$};
      %\draw (2.3,0.4) node {$B(p,r)\subset M\setminus \Omega$};
    \end{tikzpicture}
    \vskip 3em %This is here to alignt the two figures
    \caption{The interior/exterior rolling $r$-ball.}
      \end{minipage}
\end{figure}

\subsection*{Applications to integral Ricci curvature assumptions}
We present here some quantitative applications of Theorem~\ref{main1}.
Our main motivation to study extension operators relies on our interest in the Neumann heat equation of compact (sub-)manifolds.
Denote by $\Delta\geq 0$ the Laplace-Beltrami operator on a Riemannian manifold $M$. Let $u$ be a positive solution of the heat equation
\begin{align}\label{heatequation}
\partial_t u =-\Delta u, 
\end{align}
where one assumes additionally $\partial_\nu u=0$ in case $\partial M\neq \emptyset$, and $\nu$ the inward pointing normal.
If $D>0$, $K\geq 0$, then it was shown in~\cite{LiYau-86} that there are $c_1,c_2,c_3>0$ such that for any compact $M$ with convex smooth boundary, $\diam M\leq D$ and Ricci curvature bounded from below by $-K$, the solution $u$ satisfies
\begin{align}\label{gradientestimate}
c_1\vert \nabla \ln u\vert^2-\partial_t \ln u\leq c_2K+c_3t^{-1}, \quad t>0.
\end{align}
From such a gradient estimate, Li and Yau deduced a Harnack inequality as well as an upper bound for the Neumann heat kernel $h$ of $M$ of the form
\begin{align}\label{liyauupper}
h_t(x,y)\leq c_1' \Vol(B(x,\sqrt t))^{-1/2}\Vol(B(y,\sqrt t))^{-1/2}\exp\left(c_2'Kt-\frac{d(x,y)^2}{c_3' t}\right),\quad x,y\in M, t>0.
\end{align}
Inequality~\eqref{gradientestimate}, and in turn~\eqref{liyauupper}, has been generalized in~\cite{Wang-97} to compact manifolds $M$ with smooth boundary satisfying the interior rolling $r$-ball condition (cf. Definition~\ref{defrhk}~\ref{interior} and also Remark~\ref{rmk:Rsmall}), %\lookO{OK, I agree with your comment. Just have to check that the use of $R$ resp.\ $R/2$ is consistent during the article! See some other (new) footnotes \dots} 
second fundamental form bounded from below, and Ricci curvature bounded from below by $-K$, $K\geq 0$. For an extensive treatment of Neumann heat kernel estimates on non-compact domains, see, e.g., \cite{GyryaSC-11}.
\\

During the last decades there was an increasing interest in relaxing the uniform pointwise Ricci curvature lower bound to integral Ricci curvature bounds. Those provide estimates that are more stable under perturbations of the metric. In the following, we denote
\[
\rho\colon M\to\RR, \quad x\mapsto \min(\sigma(\Ric_x)),
\]
where the Ricci tensor $\Ric$ of $M$ is viewed as a pointwise
endomorphism on the cotangent bundle, and $\sigma(A)$ denotes the
spectrum of an operator $A$. For $x\in\RR$, the negative part of
$x \in \RR$ will be denoted by $x_-=\max(0,-x)\ge 0$. For a subset
$\Omega\subset M$, $p>n/2$, and $R>0$ we let
\begin{equation}\label{def:lpricci}
\kappa_{\Omega}(p,R):=\sup_{x\in \Omega}\left(\frac 1{\Vol(B(x,R))}\int_{B(x,R)}  \rho_-^p\dvol\right)^\frac{1}{p},
\end{equation}
measuring the $L^p$-mean of the negative part of Ricci curvature uniformly in balls of radius $R$ with center in $\Omega$. It is convenient to work with the scale-invariant quantity $R^2\kappa_\Omega(p,R)$. If $M$ is complete with $\partial M=\emptyset$, assuming $\kappa_{M}(p,R)$ is small for $p>n/2$ led to several analytic and geometric generalizations of results that depend on pointwise lower Ricci curvature bounds, see, e.g.,~\cite{Rose-17,DaiWeiZhang-18,Gallot-88, PetersenWei-97, PetersenWei-01, PetersenSprouse-98, Olive-19, RoseStollmann-15, Rose-17a, Aubry-07, OliveSetoWeiZhang-19, ZhangZhu-17, ZhangZhu-18}. 

If $M$ is allowed to have non-empty boundary $\partial M\neq \emptyset$ and small $\kappa_M(p,R)$, $p>n/2$, it is not known that a version of~\eqref{gradientestimate} can be derived by only adapting the techniques from~\cite{LiYau-86,Wang-97}. However, there is a way to obtain~\eqref{gradientestimate} for proper compact subsets of a manifold: 
recently, the second author~\cite{Olive-19} obtained a generalized version of~\eqref{gradientestimate} for positive solutions of~\eqref{heatequation} on compact submanifolds $\Omega\subset M$ with smooth boundary, where the volume measure of $M$ is globally volume doubling and $\kappa_\Omega(p,\diam \Omega)$ is small. Note that the latter does not imply global volume doubling unless $M$ is compact. Additionally, the obtained gradient estimate relies a priori on Gau\ss{}ian upper bounds of the Neumann heat kernel~\eqref{liyauupper}, that were obtained in~\cite{ChoulliKayserOuhabaz-15}.   Note that such gradient estimates depend implicitly on the norm of a Sobolev extension operator of $\Omega$.

We will prove a uniform, quantitative and localized version: if $\kappa_M(p,R)$ is small for some $p>n/2$, we obtain quantitative Gau\ss{}ian Neumann heat kernel upper bounds~\eqref{liyauupper} in the spirit of~\cite{ChoulliKayserOuhabaz-15} depending only on geometric conditions as well as a generalization of~\eqref{gradientestimate} for all compact $\Omega\subset M$ with boundary satisfying certain regularity conditions that does neither depend on global volume doubling nor on extension operators with unknown operator norm. 

If $M$ is complete and $\partial M=\emptyset$, we denote by $p\in\cC^\infty((0,\infty)\times M\times M)$ the heat kernel of $M$, that is, the minimal fundamental solution of~\eqref{heatequation}. If $\Omega\subset M$ is a non-empty relatively compact domain with smooth boundary $\partial\Omega\neq \emptyset$, we let $h^\Omega\in\cC^\infty((0,\infty)\times M\times M)$ be  the Neumann heat kernel, i.e., the minimal fundamental solution of~\eqref{heatequation} subject to Neumann boundary conditions on $\partial \Omega$.

Our second main theorem is the following (recall the definition of $\kappa_M(p,R)$ in~\eqref{def:lpricci}).
\begin{thm}\label{main2}
Let $M$ be a complete Riemannian manifold of dimension $n\geq 2$, $p>n/2$, $R>0$, and $K,H\geq 0$. There exists an explicitly computable $r_0=r_0(H,K)>0$ sufficiently small (cf.~Remark~\ref{rmk:Rsmall}) such that for any $r\in(0,r_0]$, there are explicitly computable constants $C=C(n,p,r,R,H,K)>0$ and $\epsilon=\epsilon(n,p,r,H,K)>0$ such that if 
\[
R^2\kappa_M(p,R)\leq \epsilon,
\]
then for any $(r,H,K)$-regular domain $\Omega\subset M$ with $\diam \Omega\leq R/2$, the Neumann heat kernel $h^\Omega$ of $\Omega$ satisfies
\begin{align}
h^\Omega_t(x,x)\leq \frac{C}{\Vol_\Omega(x,\sqrt t)}, \quad x\in\Omega,\ t>0,
\end{align}
where $\Vol_\Omega(B(x,s)):=\Vol(\Omega\cap B(x,s))$.
\end{thm}

We prove this theorem via a careful analysis of the results obtained in~\cite{ChoulliKayserOuhabaz-15, BoutayebCoulhonSikora-15} and Theorem~\ref{main1}. 

Theorem~\ref{main2} provides a tool to prove quantitative gradient estimates of type~\eqref{gradientestimate} for integral Ricci curvature assumptions. Moreover, such estimates give the opportunity to provide quantitative lower bounds on the first (non-zero) Neumann eigenvalue of $(r,H,K)$-regular subsets. Recently, the third author and G.~Wei obtained in~\cite{RoseWei-20} gradient and Neumann eigenvalue estimates assuming only the interior rolling $r$-ball condition, a lower bound on the second fundamental form, and a Kato-type condition on the negative part of the Ricci curvature defined by the Neumann heat semigroup. It is known that the latter condition is more general than assuming $\kappa_M(p,R)$ is small. We refer to~\cite{Rose-16a, Rose-19, Carron-16, CarronRose-18, RoseWei-20,RoseStollmann-18} for more information about Kato-type curvature assumptions. While those results are very general, it is hard to check that the assumptions are indeed satisfied for compact manifolds with boundary and small $\kappa_M(p,R)$. However, Theorem~\ref{main2} gives such an opportunity for relatively compact subdomains, as the following corollary shows.
\begin{cor}\label{cor:kato}
Let $M$ be a complete Riemannian manifold of dimension $n\geq 2$, $p>n/2$, $R>0$, and $K,H\geq 0$. 
There exists an explicitly computable $r_0=r_0(K,H)>0$ %small enough (cf.~Remark~\ref{rmk:Rsmall}) 
such that for any $r\in(0,r_0]$ there are explicitly computable constants $C_1=C_1(n,p,r,R,H,K)>0$, $C_2=C_2(n,p,r,R,H,K)>0$, $\epsilon=\epsilon(n,p,r,H,K)>0$ and a function $J=J_{n,p,r,R,H,K}\colon (0,\infty)\to(0,\infty)$ such that if 
\[
R^2\kappa_M(p,R)\leq \epsilon,
\]
then for any $(r,H,K)$-regular domain $\Omega\subset M$ with $\diam \Omega\leq R/2$, any positive solution of~\eqref{heatequation} with Neumann boundary conditions satisfies
\begin{align*}
J(t)\vert \nabla \ln u\vert^2-\partial_t \ln u\leq C_1+\frac{C_2}{tJ(t)}, \quad t>0.
\end{align*}
Moreover, there is an explicitly computable constant $C_3=C_3(n,p,r,R,H,K)>0$ such that if $\eta_1^\Omega$ denotes the first (non-zero) Neumann eigenvalue of $\Omega$, we have %\lookO{How does $C_4$ depend on $R$? Because we mention $R^{-2}$ explicitly in the estimate \dots \textcolor{red}{difficult question: I would guess it is so. Either I am going to check, what would need a while, or We delete $R^{-2}$.} --- Hm, interesting statement, I would like to keep it, as it is a useful and strong statement! How difficult is it? I just read a bit in \cite{RoseWei-20}, is it in Thm.~1.1? Seems it is explicit there! If we insist on explicit constants we should do so everywhere \dots}
\[
\eta_1^\Omega\geq C_3 \diam(\Omega)^{-2}.
\]
\end{cor}
\begin{remark}\label{rmk:Rsmall}
As pointed out in~\cite{Chen-90, Wang-97} the assumption that the interior rolling $r$-ball condition holds for $r>0$ small enough in Corollary~\ref{cor:kato} depends implicitly on upper bounds on the sectional curvature in $T(\partial\Omega,r)\cap\Omega$ and the lower bound on the second fundamental form $\II\geq -H$ w.r.t.\ the inward pointing normal. %\lookO{OK, I believe you. But shall we use also $r$ here for the interior rolling $r$-ball condition?}
More precisely, $r\in(0,1)$ has to be chosen such that 
\[
\sqrt{ K}\tan\left(r\sqrt{ K }\right)\leq\frac 12(1+H)\quad\text{and}\quad \frac{H}{\sqrt{ K}}\tan\left(r\sqrt{K}\right)\leq\frac 12.
\]
Thus, the restriction on the sectional curvature in the tubular neighborhood is a natural assumption. %The same reasoning applies to the curvature restrictions outside $\Omega$: the lower bound on the second fundamental form w.r.t.\ the inward pointing normal is similar to the upper bound w.r.t.\ the outward normal. The upper bound on the sectional curvature corresponds to a lower bound on the other side of $\partial\Omega$.
\end{remark}

\subsection*{Structure of the article}
The plan of this paper is as follows. In Section~\ref{section:extensions}
we construct an extension operator using Jacobi-field techniques; this construction depends on geometric assumptions that may or may not be satisfied for a given manifold. In Section~\ref{section:tubular}, we provide sufficient geometric criteria for the tubular neighborhood that imply those conditions and prove Theorem~\ref{main1}. Some proofs of auxiliary comparison estimates for hypersurfaces that we used in the proof can be found in the appendix. In Section~\ref{section:heatkernel}, we carefully adapt the main results of~\cite{BoutayebCoulhonSikora-15, ChoulliKayserOuhabaz-15} to our setting and prove Theorem~\ref{main2}. At the end, we provide proofs of Corollary~\ref{cor:kato} by showing that the Kato condition is indeed satisfied in our situation.\\

\noindent\textit{Acknowledgement:} We want to thank Leonhard Frerick, Rostislav Matveev, and J\"urgen Jost for useful remarks. C.R. wants to thank O.P. and the University of Trier for their hospitality.
\section{Quantitative Sobolev extensions}\label{section:extensions}
The existence of an extension operator $E_\Omega\colon H^1(\Omega)\to H^1(M)$ with bounded operator norm will follow by constructing an extension operator along geodesics perpendicular to $\partial\Omega$ with bounded operator norm depending on the behavior of the geodesics in a tubular neighborhood. $E_\Omega$ can then be defined on $M$ via Fermi-coordinates on $\partial\Omega$. The operator norm will then be controlled by the behavior of the volume element and the geodesics in the tubular neighborhood as explained in Section~\ref{section:tubular}.\\
Let $M=(M^n,g)$ be a Riemannian manifold (not necessarily complete) with distance function $d$. Fix an open set $\Omega\subset M$ such that $\overline \Omega\neq M$ is a smooth manifold with smooth boundary $\partial\Omega$, and denote by $\nu$ the outward normal of $\partial\Omega$. % and by $-\II$ its second fundamental form. 
Moreover, let $\dist(\cdot,\partial\Omega)$ be the distance function to $\partial\Omega$. For $r>0$, denote by $T(\partial\Omega,r)$ the $r$- tubular neighborhood of $\partial\Omega$, i.e.,
\begin{equation*}
  T(\partial\Omega,r):=\{x\in M\colon \dist(x,\partial\Omega)< r\}.
\end{equation*}
The set $T(\partial\Omega,r)$ can be parametrized by the distance function to $\partial\Omega$: %for $x\in T(\partial\Omega,r)\setminus\Omega$, set 
%\begin{equation*}
 % s:=\dist(x,\partial\Omega)
%\end{equation*}
%and for $x\in T(\partial\Omega,r)\cap\Omega$, set
%\begin{equation*}
%-s:=\dist(x,\partial\Omega).
%\end{equation*}
Set
\begin{align}
  \psi\colon (-r,r)\times \partial \Omega \to M, (s,x)\mapsto \psi(s,x)= \exp_x (\nu s).
\end{align}
Then $\psi$ satisfies for all $x\in\partial\Omega$
\[
d(\psi(s,x),x)=\begin{cases}s,& \text{if $\psi(s,x)\in M \setminus\Omega$},\\
  -s, & \text{if $\psi(s,x)\in \Omega$}.\end{cases}
\]
%where $d(\cdot,\cdot)$ denotes the Riemannian distance between points of $M$. 
The parametrization $\psi$ is defined only for $r$ small enough; more precisely, the size of $r$ depends on the focal set of $\partial\Omega$, that is, the set of points where the exponential map is non-regular, see Section~\ref{section:tubular}. We assume throughout this section that the focal set is beyond the distance $r$. %\lookO{Is the choice correct? $R$ or $R/2$?} 
In particular, there are neither focal nor cut points along any geodesic perpendicular to $\partial\Omega$. We postpone the discussion of the choice of $r$ to the next section.

Using $\psi$, the metric $g$ of $M$ decomposes into 
\[
g=ds^2+g_s, \quad s\in(-r,r),
\]
where $g_s$ is the metric of the distance hypersurface $\psi^{-1}(s,\cdot)$. We further abbreviate for $x\in T(\partial\Omega,r)$ and $s\in(-r,r)$
\[
\vert\cdot\vert_x:=\vert\cdot\vert_{g(x)}\quad\text{and}\quad \vert\cdot\vert_{s,x}:=\vert\cdot\vert_{g_s(x)}.
\]

The volume form $\dvol$ of $M$ decomposes accordingly into
\[
\dvol= ds\wedge \dvol_s,\quad s\in(-r,r),
\]
where $\dvol_s$ denotes the volume element of the distance hypersurface $\psi^{-1}(s,\cdot)$.\\

Let $U\subset M$ be open. Denote by 
$\Vert\cdot\Vert_{H^1(U)}$ the $H^1$-norm in $U$, i.e., for $u\in\cC^1(U)$, 
\[
\Vert u\Vert_{H^1(U)}^2:=\Vert u\Vert_{L^2(U)}^2+\Vert \nabla u\Vert_{L^2(U)}^2.
\]
Moreover, we denote by $\cC^1(\overline U)$ the set of all $f\in\cC^1(U)$ with continuous zeroth and first derivatives up to the boundary of $U$.\\
Throughout this section, we impose several geometric assumptions that may or may not be satisfied:
\begin{enumerate}[label=(A\arabic*)]
\item\label{ass1}There are no focal points along any geodesic $\gamma\colon(-r,r)\to M$ with $\gamma(0)\in\partial\Omega$ and $\gamma'(0)=\nu$, i.e., 
\[
\exp_x t\nu,
\]
is non-singular for any $t\in (-r,r)$. Moreover, any $\gamma$ does not hit $\partial\Omega$ twice.
\item\label{ass2} There exists $G>0$ such that for any $x\in \partial\Omega$, there exist cut-off functions $\phi_x\in \cC_{\mathrm{c}}^\infty(M)$ such that $0\leq \phi\leq 1$, $\supp\phi \subset B(x,r)$, $\phi=1$ on $B(x,r/2)$, and
\[
\vert\nabla \phi_x\vert\leq G/r.
\]
\item\label{ass3}There are functions $d,D\colon [0,r)\to (0,\infty)$ such that the volume elements $\dvol_s$ of the hypersurfaces $\psi^{-1}(s,\cdot)$ along $\gamma\colon(-r,r)\to M$, $\gamma(0)\in\partial\Omega$, $\gamma'(0)=\nu$ satisfy
\begin{align}\label{assumption2}
d(s)\dvol_s\leq \dvol_0,\quad s\in [0,r),
\end{align}
and 
\begin{align}\label{assumption22}
\dvol_0\leq D(s)\dvol_{-s},\quad s\in [0,r).
\end{align}
\item\label{ass4} the exterior and interior rolling $r$-ball condition is satisfied, cf.~Definition~\ref{defrhk}~\ref{exterior}--\ref{interior}.
\end{enumerate}
In the next section, we discuss curvature restrictions under which our assumptions are satisfied.
\begin{example}
  \label{ex:sphere}
  If $\Omega=B(0,R_0)$ is an open ball in $\RR^n$ then the above conditions are fulfilled with $r=r_0=R_0/2$ %\lookO{Check if consistent with all our settings!} 
and hence geodesics are given by $\gamma(t)=(t-R_0) \nu$ at $x \in \bd \Omega$ with $\nu=x/\Vert x \Vert$ and there is no focal point provided $\vert t \vert < r$.  Furthermore, we can choose
  \begin{equation*}
    d(s)=\Bigl(\frac{R_0}{R_0+s}\Bigr)^{n-1} 
    \qquad\text{and}\qquad
    D(s)=\Bigl(\frac{R_0}{R_0-s}\Bigr)^{n-1} .
  \end{equation*}
  Moreover, the cut-off functions $\phi_x$ can be chosen such that $G=8$.
\end{example}

For $x\in\partial\Omega$, let
\begin{equation*}
  \gamma=\gamma_x\colon (-r,r)\to M, \quad s\mapsto \exp_x(s\nu)
\end{equation*}
be the unique geodesic perpendicular to $\partial\Omega$ at $x$, and $\varphi_x$ be a cut-off function as in \ref{ass2}. Moreover, let $u\in\cC^1(\overline\Omega)$. We define the one-dimensional extension of $u$ along $\gamma$ by
\begin{align*}
(E_x u)(s):=
\begin{cases}u(\gamma(s)),& \colon s\in(-r,0],\\
 \left(-3u(\gamma(-s))+4 u(\gamma(-s/2))\right)\varphi_{x}(\gamma(s)),& 
 \colon s\in (0,r).\end{cases}
\end{align*}

\begin{prop}\label{prop:onedim}$E_x u$ is continuously differentiable along $\gamma$. Moreover, if \ref{ass1} and \ref{ass2} are satisfied, we have %\lookO{Why constants $+1$ compared to last values in proof? 165 instead 164 and 83 instead 82 \dots?}
\begin{align}
\Vert E_xu\Vert^2_{H^1(\gamma\cap(M\setminus\Omega))}&\leq 164\Vert \nabla u\Vert^2_{L^2(\gamma\cap\Omega)}
+(82+164G^2r^{-2})\Vert u\Vert^2_{L^2(\gamma\cap\Omega)}.
\end{align}
\end{prop}
\begin{proof} The continuity of $E_x u$ is obvious. Moreover, we can restrict to the case $s\geq 0$ since $E_x u$ coincides with $u$ for $s<0$. We compute the gradient of $E_xu$ by calculating its directional derivatives. The regularity of our parametrization allows to define a variation of $E_xu$ along $\gamma$ such that we can compute the partial derivatives in directions perpendicular to $\gamma$. The difficulty here is that we cannot use just a parallel frame along $\gamma$ to compute the gradient because we do not know that curves with initial tangent vectors given by the frame in small neighborhoods of the point under consideration do not intersect. Thus, we need to define an appropriate frame consisting of Jacobi fields whose curves lie in the distance hypersurfaces and that span the tangent spaces along $\gamma$. The directional derivatives in the directions given by the Jacobi fields will be obviously continuous and the gradient therefore exists.\\
For $\gamma=\gamma_x$ as above, we get for free
\begin{align}\label{partialcalc1}
\partial_s E_x u(s)&= \left(3(\partial_s u)(\gamma(-s))-2(\partial_s u)(\gamma(-s/2))\right)\varphi_{x}(\gamma(s))\\
&\quad+\left(-3u(\gamma(-s))+4 u(\gamma(-s/2))\right)(\partial_s\varphi_{x})(\gamma(s)).
\end{align}

To compute the partial derivatives perpendicular to $\gamma$, we introduce the following variation of $\gamma$:
Let $e_i\in T_xM$, $i=2,\ldots,n$, be a completion of $\nu$ to a basis. For $i \in \{2,\ldots,n\}$ let $x_i(t)$ be a curve in $\partial\Omega$ such that $x_i(0)=x$, $x'(0)=e_i$. We define the variation
\[
\gamma_i(t,s)=\exp_{x_i(t)} (s\nu), \quad t\in(-\epsilon,\epsilon), \ s\in(-r,r).
\]
Varying $t$ gives a variation through geodesics emanating perpendicularly from $\partial\Omega$. In particular, for fixed $s$, the curves $\gamma_i(\cdot,s)$ lie in the distance hypersurface with distance $s$.
The vector field
\[
X_i(s):= \frac{\partial}{\partial t} \gamma_i(t,s)\mid_{t=0}
\]
is a Jacobi field along $\gamma$ and satisfies $X_i(0)=e_i$, $X_i'(0)=S(e_i,\nu)$, where $S$ denotes the shape operator of $\partial\Omega$. In particular, we have 
\[
\langle X_i,\gamma'\rangle_{\gamma(s)}=0,\quad X_i(0)\in T_x\Omega, \quad S(X_i(0),\nu)-X'(0)=0\perp T_x\partial\Omega,
\]
so $X_i$ is a $\partial\Omega$-Jacobi field along $\gamma$ for any $i \in \{2,\ldots,n\}$. %\lookO{I changed here the range of the index $i$!!! Is it correct?  Check also elsewhere!}. 
According to~\cite{Warner-66}, $\{\nu(s)\}\cup\{X_i(s)\}_{i=2}^{n}$ spans $T_{\gamma(s)}M$.
For $i \in \{2,\dots,n\}$, define the following variation of $Eu$:
\begin{align}
  E_x^i u(t,s)
  :=\begin{cases} u(\gamma_i(t,s)),& \colon t\in(-\epsilon,\epsilon), s\in(-r,0],\\ 
\left(-3u(\gamma_i(t,-s))+4 u(\gamma_i(t,-s/2))\right)\varphi_{x}(\gamma_i(t,s)), & \colon t\in(-\epsilon,\epsilon),s\in (0,r).
\end{cases}
\end{align}
Then, for the directional derivative of $Eu(s)$ in direction $X_i(s)$, we have 
\begin{align}\label{partialcalc2}
\langle X_i,\nabla E_x u\rangle_{\gamma(s)}&= \frac{\mathrm{d}}{\mathrm{d}t} E_x^i u(t,s)\mid_{t=0}\nonumber\\
&= \frac{\mathrm{d}}{\mathrm{d}t} \left(-3u(\gamma_i(t,-s))+4 u(\gamma_i(t,-s/2))\right) \ \varphi_{x}(\gamma_i(t,s))\nonumber\\
&\quad + \left(-3u(\gamma_i(t,-s))+4 u(\gamma_i(t,-s/2))\right) \ \frac{\mathrm{d}}{\mathrm{d}t}\varphi_{x}(\gamma_i(t,s))\mid_{t=0}\nonumber\\
&= \left(-3\langle \nabla u ,\frac{\partial\gamma_i}{\partial t}\rangle_{\gamma_i(t,-s)}+4 \langle \nabla u,\frac{\partial\gamma_i}{\partial t}\rangle_{\gamma_i(t,-s/2)}\right) \ 
\varphi_{x}(\gamma_i(t,s))\nonumber\\
&\quad + \left(-3u(\gamma_i(t,-s))+4 u(\gamma_i(t,-s/2))\right) \langle \nabla\varphi_{x},\frac{\partial\gamma_i}{\partial t}\rangle_{\gamma_i(t,s)}\mid_{t=0}\nonumber\\
&= \left(-3\langle \nabla u,\frac{\partial\gamma_i}{\partial t}\mid_{t=0}\rangle_{\gamma(-s)}+4 \langle \nabla u,\frac{\partial\gamma_i}{\partial t}\mid_{t=0}\rangle_{\gamma(-s/2)}\right) \ 
\varphi_{x}(\gamma(s))\nonumber\\
&\quad + \left(-3u(\gamma(-s))+4 u(\gamma(-s/2))\right) \langle\nabla\varphi_{x},\frac{\partial \gamma_i}{\partial t}\mid_{t=0}\rangle_{\gamma(s)}\nonumber\\
&= \left(-3\langle \nabla u,X_i\rangle_{\gamma(-s)}+4 \langle \nabla u,X_i\rangle_{\gamma(-s/2)}\right) \ 
\varphi_{x}(\gamma(s))\nonumber\\
&\quad + \left(-3u(\gamma(-s))+4 u(\gamma(-s/2))\right) \langle\nabla\varphi_{x},X_i\rangle_{\gamma(s)}.
\end{align}

In particular, we have for the right limits 
\begin{align*}
\partial_s E_x u(0+)&= 3\ \partial_s u(\gamma(0))-2\ \partial_s u(\gamma(0))=\partial_s u(\gamma(0))=\partial_s u(x),\\[1.5ex]
\langle X_i, \nabla E_x u\rangle_{\gamma(0+)}&= -3\partial_i u(\gamma(0))+4\partial_i u(\gamma(0))=\partial_i u(x).
\end{align*}

It is easily seen that for the pointwise norm $E_xu$ along $\gamma$, for $s\geq 0$ we have
\begin{align}\label{l2ptw}
\vert E_x u(s)\vert^2&\leq 18 \vert u(\gamma(-s))\vert^2+32 \vert u(\gamma(-s/2))\vert^2.
\end{align}

The pointwise norm of $\nabla E_x u$ along $\gamma$ is more intricate. Since $E_x u$ coincides with $u$ in $\Omega$, we restrict the computations for the norm to $M\setminus\Omega$. 

Denote $X_1=\nu$. For any $s\in[0,r)$, we have
\begin{align}
\nabla E_xu(s)=\sum_{i=1}^n \langle \nabla E_xu,X_i\rangle_{\gamma(s)}\frac{X_i}{\vert X_i\vert_{\gamma(s)}^2}
\end{align}
and therefore
\begin{align*}
\vert\nabla E_xu\vert_{\gamma(s)}^2&=\sum_{i=1}^n \langle \nabla E_xu,X_i\rangle_{\gamma(s)}^2\\
&=\vert \partial_sE_xu(s)\vert^2+\sum_{i=2}^{n}\langle \nabla E_xu,X_i\rangle_{\gamma(s)}^2\\
&\leq \big[3\vert\partial_s u(\gamma(-s))\vert+2\vert\partial_s u(\gamma(-s/2))\vert \\
&\quad+\left(3 u(\gamma(-s))+4 u(\gamma(-s/2))\right)\vert\partial_s\varphi_{x}(\gamma(s))\vert \big]^2\\
&\quad+\sum_{i=2}^{n}\left[ \left(-3\langle \nabla u,X_i\rangle_{\gamma(-s)}+4 \langle \nabla u,X_i\rangle_{\gamma(-s/2)}\right) \ 
\varphi_{x}(\gamma(s))\right.\nonumber\\
&\quad +\left. \left(-3u(\gamma(-s))+4 u(\gamma(-s/2))\right) \langle\nabla\varphi_{x},X_i\rangle_{\gamma(s)}\right]^2\\
% &\leq 9\vert\partial_s u(\gamma(-s))\vert^2+4\vert\partial_s u(\gamma(-s/2))\vert^2\\
% &\quad+\left(9 u(\gamma(-s))^2+16 u(\gamma(-s/2))^2\right)\vert\partial_s\varphi_{x}(\gamma(s))\vert^2\\
% &\quad+2\sum_{i=1}^{n-1}\left(-3\langle \nabla u,X_i\rangle_{\gamma(-s)}+4 \langle \nabla u,X_i\rangle_{\gamma(-s/2)}\right)^2 \nonumber\\
% &\quad + 2\sum_{i=1}^{n-1}\left(-3u(\gamma(-s))+4 u(\gamma(-s/2))\right)^2 \langle\nabla\varphi_{x},X_i\rangle_{\gamma(s)}^2\\
&\leq36\vert\partial_s u(\gamma(-s))\vert^2+16\vert\partial_s u(\gamma(-s/2))\vert^2\\
&\quad+\left(36 u(\gamma(-s))^2+64 u(\gamma(-s/2))^2\right)\vert\partial_s\varphi_{x}(\gamma(s))\vert^2\\
&\quad+\sum_{i=2}^{n}36\langle \nabla u,X_i\rangle_{\gamma(-s)}^2+64 \langle \nabla u,X_i\rangle_{\gamma(-s/2)}^2 \nonumber\\
&\quad + \sum_{i=2}^{n}(36u(\gamma(-s))^2+64 u(\gamma(-s/2))^2) \langle\nabla\varphi_{x},X_i\rangle_{\gamma(s)}^2
\end{align*}
Since we are assuming that there are no focal points along $\gamma$ on the interval $(-r,r)$, as already pointed out above, the $X_i$ and $\nu$ span $T_{\gamma(s)}M$ for all $s\in(-r,r)$. Hence,
\begin{align*}
\vert\nabla E_xu\vert_{\gamma(s)}^2&\leq 36\vert\nabla u\vert_{\gamma(-s)}^2+64\vert\nabla u\vert_{\gamma(-s/2)}^2\\
&\quad+\left(36u(\gamma(-s))^2+64u(\gamma(-s/2))^2\right)\vert\nabla\varphi\vert_{\gamma(s)}^2.
\end{align*}
Denote by $\chi_I$ the characteristic function of $I\subset \RR$. Using (A3), we get
\begin{align}\label{normptw}
\vert\nabla E_xu\vert_{\gamma(s)}^2&\leq 36\vert\nabla u\vert_{\gamma(-s)}^2+64\vert\nabla u\vert_{\gamma(-s/2)}^2\nonumber\\
&\quad+\left(36u(\gamma(-s))^2+64u(\gamma(-s/2))^2\right)\chi_{[r/2,r]} G^2r^{-2}.
\end{align}
%\lookO{\label{fn:extra-factor}\textcolor{magenta}{I don't see where the factor $s^{-2}$ is coming from! There is no derivative on $u$, but only on $\phi_x$, hence the factor $G^2$.  This would improve our result, as I don't think $G^2/R^2$ is optimal! Could you check???}}

Now we can compute the $H^1$-norm of $E_xu$.
We already mentioned above that we only care about the norm in the complement of $\Omega$. By~\eqref{l2ptw}, for the $L^2$-norm of $E_xu$, we have
\begin{align}\label{l2geodesic}
\Vert E_x u\Vert^2_{L^2(\gamma\cap (M\setminus\Omega))}
&=\int_0^{r} \vert E_x u(s)\vert^2\drm s\nonumber\\
&\leq
\int_0^{r} 18 \vert u(\gamma(-s))\vert^2+32 \vert u(\gamma(-s/2))\vert^2\drm s\nonumber\\
&\leq \int_{-r}^0 
18\vert u(\gamma(s))\vert^2+32\vert u(\gamma(s/2))\vert^2\drm s\nonumber\\
&= \int_{-r}^0 
18\vert u(\gamma(s))\vert^2\drm s+\int_{-r/2}^0 64\vert u(\gamma(s))\vert^2\drm s\nonumber\\
&\leq 82\int_{-r}^0 \vert u(\gamma(s))\vert^2\drm s\nonumber\\
&=82 \Vert u\Vert_{L^2(\gamma\cap\Omega)}^2
\end{align}
By~\eqref{normptw}, for the $L^2$-norm of $\nabla E_xu$, we have
\begin{align*}
\Vert \nabla E_x u\Vert_{L^2(\gamma\cap (M\setminus\Omega))}^2
&= \int_0^{r} \vert \nabla E_x u(s)\vert_{\gamma(s)}^2\drm s\\
&\leq \int_0^{r}36\vert\nabla u\vert_{\gamma(-s)}^2+64\vert\nabla u\vert_{\gamma(-s/2)}^2\\
&\quad+\left(36u(\gamma(-s))^2+64u(\gamma(-s/2))^2\right)\chi_{[r/2,r]} G^2r^{-2}\drm s\\
&\leq \int_0^{r}36\vert\nabla u\vert_{\gamma(-s)}^2\drm s+\int_0^{r/2}128\vert\nabla u\vert_{\gamma(-s)}^2\drm s\\
&\quad+G^2r^{-2}\left(\int_{r/2}^{r}36u(\gamma(-s))^2\drm s+\int_{r/4}^{r/2}128u(\gamma(-s))^2\drm s\right)\\
&\leq 164\Vert \nabla u\Vert_{L^2(\gamma\cap\Omega)}^2+164G^2r^{-2}\Vert u\Vert_{L^2(\gamma\cap\Omega)}^2.
\end{align*}
Hence, the claim follows.
\end{proof}

We are now going to define the extension operator $E_\Omega$.
Denote by $x'$ the distance minimizing point of $x\in T(\partial \Omega, r)\setminus \overline\Omega$ to $\partial\Omega$. This point always exists and is unique due to our uniquely defined parametrization. For $u\in\cC^1(\overline\Omega)$, we define the extension $E_\Omega u$ by 
\begin{align*}
E_\Omega u(x)
  :=\begin{cases}
    u(x),& \text{if $x\in\overline\Omega$},\\
    E_{x'}u(d(x,x')),& \text{if $x\in T(\partial \Omega, r)\setminus \overline\Omega$}, \\
    0, & \text{otherwise.}
  \end{cases}
\end{align*}

The next theorem shows that this operator is indeed a map $E_\Omega\colon H^1(\Omega)\to H^1(M)$ with bounded operator norm under assumptions (A1), (A2), and (A3).
\begin{thm}\label{thm:medi}
Assume \ref{ass1}, \ref{ass2}, \ref{ass3}, and \ref{ass4} . Then, the operator $E_\Omega$ defined above is linear and continuous from $H^1(\Omega)$ to $H^1(M)$ with operator norm
\begin{align}
\Vert E_\Omega\Vert^2\leq 1+\max_{s,t\in[0,r]}\frac {D(t)}{d(s)}\max(164,82+164G^2r^{-2}).
\end{align}
\end{thm}
\begin{proof} We can restrict our considerations to $u\in\cC^1(\overline\Omega)$ since $\cC^1(\overline\Omega)$ is dense in $H^1(\Omega)$. Assumption \ref{ass4} implies that no geodesics starting in $\partial\Omega$ and emanating perpendicularly meet in the interior or exterior of $\Omega$. To compute the operator norm of $E_\Omega$, note that
\begin{align}\label{sumnorm}
\Vert E_\Omega u\Vert^2_{H^1(M)}= \Vert u\Vert^2_{H^1(\Omega)}+\Vert E_\Omega u\Vert_{H^1(T(\partial\Omega,r)\setminus\Omega)}^2.
\end{align}
Assumption~\eqref{assumption2} on the volume element of the hypersurfaces on the interval $[0,r)$ implies
\begin{align*}
\Vert E_\Omega u\Vert_{H^1(T(\partial\Omega,r)\setminus\Omega)}^2
&=\Vert E_\Omega u\Vert_{L^2(T(\partial\Omega,r)\setminus \Omega)}^2+\Vert \nabla E_\Omega u\Vert_{L^2(T(\partial\Omega,r)\setminus\Omega)}^2\nonumber\\
&=\int_0^{r} \int_{\partial\Omega}\left(\vert E_\theta u (s)\vert^2+\vert \nabla E_\theta u(s)\vert^2_{\gamma(s)}\right)\dvol_{s}\theta\drm s\nonumber\\
&\leq \int_0^{r} \int_{\partial\Omega}\left(\vert E_\theta u (s)\vert^2+\vert \nabla E_\theta u(s)\vert^2_{\gamma(s)}\right)\frac 1{d(s)}\dvol_0\theta\drm s\nonumber\\
&=\int_{\partial\Omega}\int_0^{r} \left(\vert E_\theta u (s)\vert^2+\vert \nabla E_\theta u(s)\vert^2_{\gamma(s)}\right)\frac 1{d(s)}\drm s\dvol_0\theta\nonumber\\
&\leq \max_{s\in[0,r]}\frac 1{d(s)}\int_{\partial\Omega}\Vert E_\theta u\Vert^2_{H^1(\gamma_\theta\cap M\setminus \Omega)}\dvol_0\theta\nonumber
\end{align*}
Using Proposition~\ref{prop:onedim}, the last integral can be interpreted as an integral over $T(\partial\Omega,r)\cap\Omega$, and from~\eqref{assumption22} we get
\begin{align*}
\Vert E_\Omega u\Vert_{H^1(T(\partial\Omega,r)\setminus\Omega)}^2
&\leq  \max_{s\in[0,r]}\frac {164}{d(s)}\int_{\partial\Omega}\Vert \nabla u\Vert^2_{L^2(\gamma_\theta\cap\Omega)}\dvol_0\theta\\
&\quad
+(82+164G^2r^{-2})\max_{s\in[0,r]}\frac {1}{d(s)}\int_{\partial\Omega}\Vert u\Vert^2_{L^2(\gamma_\theta\cap\Omega)}\dvol_0\theta\\
&\leq  \max_{s\in[0,r]}\frac {164}{d(s)}\int_{\partial\Omega}\Vert \nabla u\Vert^2_{L^2(\gamma_\theta\cap\Omega)}D(s)\dvol_{-s}\theta\\
&\quad
+(82+164G^2r^{-2})\max_{s\in[0,r]}\frac {1}{d(s)}\int_{\partial\Omega}\Vert u\Vert^2_{L^2(\gamma_\theta\cap\Omega)}D(s)\dvol_{-s}\theta\\
&\leq  164\max_{s,t\in[0,r]}\frac {D(t)}{d(s)}\int_{\partial\Omega}\Vert \nabla u\Vert^2_{L^2(\gamma_\theta\cap\Omega)}\dvol_{-s}\theta\\
&\quad
+(82+164G^2r^{-2})\max_{s,t\in[0,r]}\frac {D(t)}{d(s)}\int_{\partial\Omega}\Vert u\Vert^2_{L^2(\gamma_\theta\cap\Omega)}\dvol_{-s}\theta\\
&=164\max_{s,t\in[0,r]}\frac {D(t)}{d(s)}\Vert \nabla u\Vert_{L^2(\Omega\cap T(\partial\Omega,r))}^2\\
&\quad+(82+164G^2r^{-2})\max_{s,t\in[0,r]}\frac {D(t)}{d(s)}\Vert u\Vert_{L^2(\Omega\cap T(\partial\Omega,r))}^2.
\end{align*}
Hence,~\eqref{sumnorm} becomes
\begin{align*}
\Vert E_\Omega u\Vert^2_{H^1(M)}&\leq \left(1+ 164\max_{s,t\in[0,r]}\frac {D(t)}{d(s)}\right)\Vert \nabla u\Vert_{L^2(\Omega)}^2\\
&\quad+\left(1+(82+164G^2r^{-2})\max_{s,t\in[0,r]}\frac {D(t)}{d(s)}\right)\Vert u\Vert_{L^2(\Omega)}^2,
\end{align*}
which implies the theorem.
\end{proof}
\begin{example}
  \label{ex:sphere2}
  If $\Omega=B(0,R_0)$ is an open ball in $\RR^n$ (see Example~\ref{ex:sphere}) then we have
  \begin{equation*}
    \max_{s,t\in[0,r]}\frac {D(t)}{d(s)}
    =\left(\frac{1+r/R_0}{1-r/R_0}\right)^{n-1}
  \end{equation*}
  and this constant is bounded by $3^{n-1}$ for $R \le R_0$.
  Moreover, the norm of the extention operator then is bounded by
  \begin{equation*}
    \Vert E_\Omega\Vert^2
    \leq 1+ 3^{n-1}\max(164,82+1312 r^{-2}).
  \end{equation*}
\end{example}
\section{Geometry of tubular neighborhoods around hypersurfaces and the proof of Theorem~\ref{main1}}\label{section:tubular}

In order to prove Theorem~\ref{main1}, we need geometric assumptions in order to bound the norm of the extension operator obtained in Theorem~\ref{thm:medi}, which depends crucially on the geometry of the tubular neighborhood of the boundary of the considered domain.
Let $M$ be a Riemannian manifold of dimension $n\geq 2$, and $N\subset M$ a hypersurface in $M$. In order to provide a quantitative estimate depending on curvature restrictions and the size of the tubular neighborhood, it is necessary to justify the regularity of the coordinate maps we use in our calculations. Our aim is to prove that under certain curvature restrictions in the tubular neighborhood of $N$, there exists $r>0$ such that for any $p\in N$, 
\begin{equation*}
  \exp_pt\nu,\quad t\in (-r,r)
\end{equation*}
is non-singular. On one hand, $\exp$ is non-singular only up to the cut-locus. On the other hand, geodesics emanating from different points starting in $N$ must not intersect. If such two geodesics intersect in a point, this point is called a focal point. It is known that focal points appear not later than cut points~\cite{Warner-66}. Thus, we can restrict our investigation to the absence of focal points along any geodesic emanating from $N$.  Below we provide estimates for the focal distance along a geodesic. Note that those are local considerations which do not focus on geodesics with starting points lying far away from each other inside the boundary.\\

Denote by $\Sec$ the sectional curvature of $M$ and $\II$ the second fundamental form of $N$ w.r.t.~the outward pointing normal $\nu$. Let $\gamma\colon [0,l]\to M$ be a distance minimizing geodesic such that $\gamma(0)\in N$, $\gamma'(0)\in (T_{\gamma(0)}N)^\perp$. As explained above, a point $q\in \gamma$ is called a \emph{focal point} if the exponential map at $\gamma(0)$ is singular in $q$. The absence of focal points is given by the following.

\begin{lem}[{\cite[Corollary~4.2]{Warner-66}}]\label{lemfocaldistance} Let $M$, $N$, $\gamma$ as above and $H,K\in\RR$. %\begin{enumerate}%[(i)]
%\item 
Suppose $\II\geq H$ in $\gamma(0)$ and $\Sec\leq K$. Then, there are no focal points along $\gamma$ on $[0,\min(r_0,l))$, where 
$r_0>0$ is the smallest positive number $r$ such that one of the three conditions below is satisfied:
\begin{align}\label{focaldistance}
\begin{cases}
\cot (\sqrt Kr)=\frac H{\sqrt K}, & \text{if $K>0$},\\
r=\frac 1H, & \text{if $K=0$},\\
\coth (\sqrt Kr)=\frac H{\sqrt K} & \text{if $K<0$}.
\end{cases}
\end{align}
%\item Suppose $\II\leq H$ in $\gamma(0)$ and $\Sec\geq K$, then there are no focal points along $\gamma$ on $[0,\min(r_0,l)]$ with $r_0$ as above.
%\end{enumerate}
\end{lem}
Note that the equations~\eqref{focaldistance} come from a comparison result with the first zero of the $N$-Jacobi field equation in a space of constant curvature for some hypersurface with constant second fundamental form. If no positive solution exists, there are no focal points along any geodesic.
Warner's result ensures that condition \ref{ass1} is satisfied, so that the parametrization of the tubular neighborhood of $\partial \Omega$ is well-defined for $r_0$. Now we will provide conditions such that \ref{ass2} is satisfied in an appropriate neighborhood of $N$. Suppose that $\gamma\colon [0,r_0)\to M$, $\gamma(0)=p\in N$ is as above. We need to control the metric tensor along $\gamma$ in terms of geometric restrictions. To this end, we provide a comparison estimate for the distance hypersurfaces along $\gamma$. Although the proof is straightforward and adapted from the proof in~\cite{Petersen-06} for distance spheres, we are not aware of any result in the literature that gives our comparison estimates below, so we decided to include a full proof in the appendix using the Riccati comparison technique.
As in the section above, we parametrize the metric with respect to the distance function to $N$, $s:=\dist(\cdot,N)$. Then, we can decompose
\begin{align}\label{metricdecomposition}
g=ds^2+g_s,
\end{align}
where $g_s$ is the metric on $N$ evolving with respect to $s$ as long as there are no focal points along the corresponding distance minimizing geodesic $\gamma$ perpendicular to $N$ up to $r_0$. 
\begin{prop}\label{hessiancomparison}
Let $k, K\in\RR$, and $\gamma$, $r_0$ as above. If $\Sec\geq k$ along $\gamma$ on the interval $[0,r_0)$, then for almost all $s\in[0,r_0)$,
\begin{equation}\label{lowerbound}
g_s\leq \mu_{k,H_+}^2(s)g_0,
\end{equation}
and if $\Sec\leq K$, then for almost all $s\in[0,r_0)$, 
\begin{equation}
\mu_{K,H_-}^2(s)g_0\leq g_s,
\end{equation}
where $H_+$ and $H_-$ are the maximum resp.\ minimum principal curvatures of $N$ in $\gamma(0)$, and $\mu_{k,H}(s)$ are functions %\lookO{Can we make them precise, and say, in the next proposition, $d(s)=\mu_{k,H_+}(s)^{n-1}$ etc.?} 
that arise from the solution of a Ricatti type ODE (cf.~Appendix~\ref{appendix}, Eq.~\eqref{substitution}).
\end{prop}

\begin{prop}\label{prop:volumecomparison}
Under the assumptions as in Proposition~\ref{metricdecomposition}, we have
\begin{align}\label{volumecomparison}
\dvol_s&\leq \dvol^{k,H_+}_s=D(s)\dvol_0, \quad s\in [0,r_0), \\
\quad\text{and}\quad 
\dvol_{-s}&\geq \dvol^{K,H_-}_{-s}=d(-s)\dvol_0, \quad s\in [0,r_0).
\end{align}
with
\[
D(s)=\mu_{k,H_+}^{n-1}(s),\quad\text{and}\quad d(s)=\mu_{K,H_-}^{n-1}(s).
\]
\end{prop}

For a self-contained proof, see the Appendix. An additional ingredient for the upper bound of the extension operator is the existence of cut-off functions with controlled gradient, i.e., condition~\ref{ass2}. %This is in fact one of the assumptions from Theorem~\ref{main1}, so we could proceed directly to its proof. 
However, we have to assume that the sectional curvature is bounded from below in the tubular neighborhood of $\Omega$ to compare the metric tensors along geodesics perpendicular to $\partial\Omega$. This readily implies Laplace comparison for the distance function, and hence, the existence of appropriate cut-off functions, see, e.g., \cite{LiYau-86}, where cut-off functions are constructed depending on lower bounds of the Ricci curvature.

%it is particularly interesting to consider situations where the constant $G$ that controls the cut-off function can be estimated in geometric terms. This is the case, for example, in the presence of a uniform Ricci curvature lower bound. For our application to heat kernel estimates depending on $L^p$-curvature bounds, we recall the following result from \cite{DaiWeiZhang-18}, %\lookO{More precisely: \cite[Lemma~5.3]{DaiWeiZhang-18} seems only to cite the result \cite[Theorem~6.4]{PetersenWei-01} without proof, am I right? So why citing? Also, can we make the dependence of $G$ on $R$ more explicit? We say that we give \emph{explicit} constants, but here we are cheating \dots probably an explicit $\epsilon$ would also be nice!}
%which implies the existence of bounds on cut-off functions assuming only an integral Ricci curvature condition.
%\begin{lem}[{\cite[Lemma~5.3]{DaiWeiZhang-18}},{\cite[Theorem~6.4]{PetersenWei-01}}]\label{daiweizhang}
%Suppose $n\geq 2$, $p>n/2$, $R>0$. There is an $\epsilon>0$ such that if
%\[
%R^2\kappa_M(p,R)\leq \epsilon,
%\]
%then, for any $x\in M$, $r\in(0,R]$, there exists $\phi\in\cC_{\mathrm{c}}^\infty(B(x,r))$ and a constant $G>0$ depending on $n$ $p$, and $R$ such that 
%\[
%\supp\phi\subset B(x,r),\quad \phi\mid_{B(x,r/2)}=1, \quad \text{and}\quad \vert\nabla\phi\vert\leq G/{d(x,\cdot)}.
%\]
%\end{lem}
%
%It should be noticed that the formulation of \cite[Lemma~5.3]{DaiWeiZhang-18} is different from our claim. However, looking at the proof of \cite[Theorem~6.4]{PetersenWei-01} shows that Lemma~\ref{daiweizhang} is the correct statement.

\begin{proof}[Proof of Theorem~\ref{main1}] According to the explanation regarding the existence of cut-off functions, \ref{ass2} is satisfied by Definition~\ref{defrhk}~\ref{main:tubular}. If $\Omega\subset M$ satisfies \ref{exterior}, \ref{interior}, \ref{main:secondf} and \ref{main:tubular} of Definition~\ref{defrhk}, then according to Lemma~\ref{lemfocaldistance}, \ref{ass1} and \ref{ass3} are satisfied for $r$ being the minimum of all the $r_0$ obtained in~\eqref{focaldistance} for any point in $\partial\Omega$.
%the minimum of the $r_0$ obtained in~\eqref{focaldistance} and the maximum $r$ appearing in Remark~\ref{rmk:Rsmall}. 
Thus, the extension operator constructed in Section~\ref{section:extensions} exists and is bounded uniformly in $K$ and $H$.
\end{proof}
\begin{remark}
The extension operator defined above can be used in~\cite{Olive-19} to obtain a Li-Yau type estimate on the Neumann heat kernel for $(r,H,K)$-regular domains under integral curvature conditions with appropriately chosen $r>0$, where the estimate would only depend on geometric parameters. However, one would still need to require the ambient space to be globally doubling, since in general only local volume doubling holds under integral Ricci curvature assumptions. To remove the necessity of this condition, we develop in the next section local estimates for the Neumann heat kernel that will only require local volume doubling. 
\end{remark}

\section{Localizing estimates for the Neumann heat kernel}\label{section:heatkernel}

In~\cite{ChoulliKayserOuhabaz-15}, the authors observed that the main result of~\cite{BoutayebCoulhonSikora-15} can be used to prove full Gau\ss{}ian upper bounds for the Neumann heat kernel of relatively compact domains wiht Lipschitz boundary satisfying the volume doubling property provided the ambient space is globally volume doubling and its heat kernel has full Gau\ss{}ian upper bounds. In the proof they use the existence of a bounded extension operator. Although they show as an example that their result holds for relatively compact domains with Lipschitz boundary in $\RR^n$ and hyperbolic spaces, they do not discuss conditions for the extension operator to be uniformly bounded in geometric parameters. In Section~\ref{section:tubular} we have shown the existence of extension operators with this property, such that those extension operators can be used to give a quantitative version of the results in \cite{ChoulliKayserOuhabaz-15}.

Furthermore, the paper does not discuss any localized results, i.e., if local volume doubling and small-time upper bound for the heat kernel on $M$ imply a small-time Gau\ss{}ian upper bound for the Neumann heat kernel on a domain $\Omega\subset M$ with sufficiently regular boundary. This is of independent interest for other quantitative applications. In fact, their main theorem is based on \cite[Theorem 1.1]{BoutayebCoulhonSikora-15}, which does not apply to small-time Gau\ss{}ian upper bounds. While it is indicated on p.~308 %]{BoutayebCoulhonSikora-15} 
of the latter article that their results hold in the localized situation as well, \cite[Theorem 1.1]{BoutayebCoulhonSikora-15}, as it is formulated, does not hold assuming only local instead of global volume doubling. 

We will show here Neumann heat kernel upper bounds for $(r,H,K)$-regular bounded subsets of $M$ while $R^2\kappa_M(p,R)$ is small for $p>n/2$, i.e., Theorem~\ref{main2}. As explained above, this does neither follow directly from~\cite{ChoulliKayserOuhabaz-15} nor from~\cite{BoutayebCoulhonSikora-15}. %Thus, Theorem~\ref{main2} will follow from those considerations and our explicit estimate for the extension operator. . 

The results presented below are adaptions from~\cite{BoutayebCoulhonSikora-15} and~\cite{ChoulliKayserOuhabaz-15}. Altough the proofs are almost the same, we give a complete outline, because the differences are quite subtle. The main differences are that we only use local volume doubling and localized heat kernel estimates and our extension operators from Theorem~\ref{main1}. First, we will show that a local volume doubling and upper heat kernel estimate imply a family of localized Gagliardo-Nirenberg inequalities. Although we work directly on Riemannian manifolds, the proof is the same for metric measure spaces. Second, we show that Theorem~\ref{main1} implies local Gagliardo-Nirenberg inequalities on $(r,H,K)$-regular bounded domains with appropriate $r$ with a proof slightly different from~\cite{ChoulliKayserOuhabaz-15}. The desired upper Neumann heat kernel bound then follows directly from~\cite[Theorem~1.1]{BoutayebCoulhonSikora-15}, since the Neumann Laplace operator satisfies the finite speed propagation property.

In the following we use a notation similar to \cite{BoutayebCoulhonSikora-15}. Let $M$ be a Riemannian manifold of dimension $n\geq 2$. Denote $v_r(x):=\Vol(B(x,r))$. We say that $M$ satisfies the \emph{local volume doubling condition} if there are $C_D>0$ and $R>0$ such that
\begin{align}\label{doubling}
v_r(x)\leq C_D\left(\frac rs\right)^n v_s(x),\quad x\in M, \ 0<s\leq r\leq R.
\end{align}

We say that the heat kernel $p$ of $M$ satisfies \emph{local upper bounds} if there are $C,R>0$ such that
\begin{align}\label{hkestimate}
p_t(x,y)\leq \frac{C}{v_{\sqrt t}(x))^\frac{1}{2}v_{\sqrt t}(y)^\frac{1}{2}}, \quad t\in (0,R^2/4].
\end{align}

Furthermore, we define the following family of Gagliardo-Nirenberg inequalities: there exists $R_0>0$, $q\in(2,\infty]$, $\frac{q-2}qn<2$, such that
%\lookO{$\Delta$ is always the (non-negative) Laplacian in this article, yes?}
\begin{align}\label{GN}
  \Bigl\Vert v_r^{\frac 12-\frac 1q} f\Bigr\Vert_q\leq C \left(\Vert f\Vert_2+r\Vert \Delta^{1/2}f\Vert_2\right), \quad r\leq R_0.
\end{align}
%\lookO{I think you should give a lable like ($vEv_{p,q,\gamma}$) here!}

Inequalities~\eqref{hkestimate} and~\eqref{GN} are short-time and localized assumptions in contrast to the global assumptions made in~\cite{BoutayebCoulhonSikora-15}. One could also choose another function $v$ instead of the volume measure satisfying similar properties as in~\eqref{doubling} and~\eqref{hkestimate}, while the underlying volume measure fulfils~\eqref{doubling} separately, but we decided not to do so for the sake of presentation.

We will often write $f$ to denote the multiplication operator by the function $f$. Given $1\leq p,q\leq +\infty$ and $\gamma, \delta\geq 0$ such that $\gamma + \delta = \frac{1}{p}-\frac{1}{q}$, we will say that the condition $(vEv_{p,q,\gamma})$ holds if there exist $t_0>0$ such that
\begin{equation}\label{vEv}\tag{$vEv_{p,q,\gamma}$}
  \sup_{0<t\leq t_0} \|v_{\sqrt{t}}^\gamma \e^{-t\Delta}v_{\sqrt{t}}^\delta \|_{p,q} <\infty,
\end{equation}
where $\|\cdot \|_{p,q}$ denotes the operator norm from $L^{p}(M)$ to $L^q(M)$, i.e., if there is a uniform bound on $(0,t_0]$ for the operator norm of $v_{\sqrt{t}}^\gamma \e^{-t\Delta}v_{\sqrt{t}}^\delta$. As with the assumption~\eqref{hkestimate}, the difference between our assumption $(vEv_{p,q,\gamma})$ and the one in~\cite{BoutayebCoulhonSikora-15} is that our assumption is local, for short time, as opposed to a global assumption for all values of $t>0$.  Note that if $p'$ and $q'$ are the conjugate exponents of $p$ and $q$, respectively, then by duality $(vEv_{p,q,\gamma})$ is equivalent to $(vEv_{p',q',\delta})$. 

\begin{prop}[{cf.~\cite[Proposition~2.1.1 and Corollary~2.1.2]{BoutayebCoulhonSikora-15}}]\label{hknorm}
 Assume that~\eqref{doubling} is satisfied. Then the following conditions are equivalent.%\lookO{I find that the statement was hard to understand? The condition $(vEv_{1,\infty,\frac{1}{2}})$ is hard to find in the text. Is the notation standard?  I see, in \cite[Proposition~2.1.1 and Corollary~2.1.2]{BoutayebCoulhonSikora-15} it is written even worse :-(}
 \begin{itemize}
 
\item~\eqref{hkestimate} holds up to $t_0=R^2/4$,
\item $(vEv_{\infty,\infty,\frac12})$ %\lookO{correct interpretation???}
  is satisfied up to time $t_0$,
  \item $(vEv_{1,\infty,\frac{1}{2}})$ (i.e.,~\eqref{vEv} with $p=1$, $q=\infty$ and $\gamma=1/2$) is satisfied up to time $t_0$,
  \item $(vEv_{1,2,0})$  (i.e.,~\eqref{vEv} with $p=1$, $q=2$ and $\gamma=0$) is satisfied up to time $t_0/2$,
  \item $(vEv_{2,\infty,\frac12})$  (i.e.,~\eqref{vEv} with $p=2$, $q=\infty$ and $\gamma=1/2$) is satisfied up to time $t_0/2$.
 \end{itemize}
\end{prop}

\begin{proof}
The proof is analogous to the one of Proposition~2.1.1 and Corollary~2.1.2 in~\cite{BoutayebCoulhonSikora-15}. The only difference is that the last two conditions will hold for different time ranges than the first two. The equivalence between~\eqref{hkestimate} and $(vEv_{1,\infty,\frac{1}{2}})$ follows directly from the Dunford-Pettis theorem: for all $t\in(0,t_0]$, we have
\[
\Vert v_{\sqrt t}^{1/2}\euler^{-t\Delta}v_{\sqrt t}^{1/2}\Vert_{1,\infty}=\sup_{x,y\in M} v_{\sqrt t}^{1/2}(x)p_t(x,y)v_{\sqrt t}^{1/2}(y)\leq C<\infty.
\]
Since the proof in~\cite{BoutayebCoulhonSikora-15} is for a fixed value of $t$, the same proof gives us the equivalence in our case. The equivalence between $(vEv_{1,2,0})$ and $(vEv_{2,\infty,\frac{1}{2}})$ follows by duality. It suffices to show that $(vEv_{1,\infty, \frac{1}{2}})$ holds up to time $t_0$ if and only if $(vEv_{1,2,0})$ holds up to time $t_0/2$. Note that if $T:L^1 \rightarrow L^2$, then
\[\|T^*T\|_{1,\infty} = \|T^*\|_{2,\infty}^2 = \|T\|_{1,2}^2,\]
so by taking $T=\e^{-(t/2)\Delta}v^{1/2}_{\sqrt{t}}$, we get that
\[\|v_{\sqrt{t}}^{1/2}\e^{-t\Delta}v^{1/2}_{\sqrt{t}}\|_{1,\infty} = \|v^{1/2}_{\sqrt{t}}\e^{-(t/2)\Delta}\|_{2,\infty}^2 = \|\e^{-(t/2)\Delta}v^{1/2}_{\sqrt{t}}\|_{1,2}^2.\]
 Hence, $(vEv_{1,\infty,\frac{1}{2}})$ is equivalent to 
 \[\sup_{0<t\leq t_0}\|v^{1/2}_{\sqrt{t}}\e^{-(t/2)\Delta}\|_{2,\infty}^2 <\infty,\]
 as well as to 
 \[\sup_{0<t\leq t_0}\|\e^{-(t/2)\Delta}v^{1/2}_{\sqrt{t}}\|_{1,2}^2 <\infty.\]
 If $(vEv_{1,\infty,\frac{1}{2}})$ holds up to time $t_0$, then defining $\tilde{t} = 2t$ for any $0<t\leq t_0/2$, we have that
 \[\sup_{0<t\leq t_0/2}\|\e^{-t\Delta}v^{1/2}_{\sqrt{t}}\|_{1,2}= \sup_{0<\tilde{t}\leq t_0}\|\e^{-(\tilde{t}/2)\Delta}v^{1/2}_{\sqrt{\tilde{t}/2}}\|_{1,2} \leq  \sup_{0<\tilde{t}\leq t_0}\|\e^{-(\tilde{t}/2)\Delta}v^{1/2}_{\sqrt{\tilde{t}}}\|_{1,2} < +\infty \]
 where we used that $v_r(x)$ is non-decreasing in $r$. Thus $(vEv_{1,2,0})$ holds up to time $t_0/2$. Conversely, if $(vEv_{1,2,0})$ holds up to time $t_0/2$, then we have
 \[\sup_{0<\tilde{t}\leq t_0}\|\e^{-(\tilde{t}/2)\Delta}v^{1/2}_{\sqrt{\tilde{t}}}\|_{1,2} = \sup_{0<t\leq t_0/2}\|\e^{-t\Delta}v^{1/2}_{\sqrt{2t}}\|_{1,2} \leq \sup_{0<t\leq t_0/2} C\|\e^{-t\Delta}v^{1/2}_{\sqrt{t}}\|_{1,2} <+\infty,\]
 where we used that $v_{\sqrt{2t}}(x) \leq v_{2\sqrt{t}}(x) \leq Cv_{\sqrt{t}}(x)$ by the non-decreasing and the $v$-doubling~\eqref{doubling} properties of $v_r(x)$. Hence $(vEv_{1,\infty, \frac{1}{2}})$ holds up to time $t_0$, and this completes the proof.  
\end{proof}

\begin{prop}
Assume that $M$ satisfies~\eqref{doubling} and~\eqref{hkestimate}. Moreover, let $q\in(2,\infty]$, $\frac{q-2}qn<2$. Then~\eqref{GN} holds with $R_0=R/2$.
\end{prop}
\begin{proof}
The proof is adapted from the beginning of Section~2 and Proposition~2.3.2 of~\cite{BoutayebCoulhonSikora-15}. According to Proposition~\ref{hknorm},~\ref{doubling} and~\ref{hkestimate} for $R>0$ implies the existence of a $C>0$ such that
\begin{align}\label{normest}
H:=\sup_{t\in(0,R/2]} \Vert v_{\sqrt t}^{\frac 12-\frac 1q}\euler^{-t\Delta}\Vert_{2,q}\leq C.
\end{align}
The fundamental theorem gives for any $f\in L^2$
\[
f=\euler^{-t\Delta} f+\int_0^t \Delta\euler^{-s\Delta} f\drm s.
\]
Thus, putting $\alpha=\frac 12-\frac 1q$,
\begin{align*}
\Vert v_{\sqrt t}^\alpha f\Vert_q&\leq \Vert v_{\sqrt t}^\alpha \euler^{-t\Delta}\Vert_{2,q}\Vert f\Vert_2+\int_0^t \Vert v_{\sqrt t}^\alpha \euler^{-s\Delta/2}\Vert_{2,q}\Vert \Delta\euler^{-s\Delta/2}f\Vert_2 \drm s.
\end{align*}
Using~\eqref{doubling}, we get 
\begin{align*}
\Vert v_{\sqrt t}^\alpha f\Vert_q
&\leq \Vert v_{\sqrt t}^\alpha \euler^{-t\Delta}\Vert_{2,q}\Vert f\Vert_2+\int_0^t \Vert \frac{v_{\sqrt t}}{v_{\sqrt{s/2}}} \Vert_\infty^\alpha\Vert v_{\sqrt{s/2}}^\alpha \euler^{-s\Delta/2}\Vert_{2,q}\Vert \Delta\euler^{-s\Delta/2}f\Vert_2 \drm s\\
&\leq H\Vert f\Vert_2 +C_D2^{n/2}\int_0^t \left(\frac ts\right)^{n\alpha/2}\Vert \Delta^{1/2}\euler^{-s\Delta/2}\Delta^{1/2}f\Vert_2\drm s\\
&\leq H\Vert f\Vert_2 +GH t^{n\alpha/2}\int_0^t s^{-n\alpha/2-1/2}\Vert \Delta^{1/2}f\Vert_2\drm s,
\end{align*}
and the last integral is finite.
Hence, for all $\sqrt t\in(0,R/2]$,
\begin{align*}
\Vert v_{\sqrt t}^\alpha f\Vert_q\leq C(\Vert f\Vert_2 +\sqrt t\Vert \Delta^{1/2}f\Vert_2)
\end{align*}
for some $C$ depending only on $q$, $C_D$, and $n$. Putting $r=\sqrt t$ yields the result.
\end{proof}

\begin{cor}\label{cor:GN}Suppose $R>0$, and $2p>n\geq2$. There is an $\epsilon>0$ such that if a manifold $M$ of dimension $n$ satisfies
\[
\kappa_M(p,R)\leq \epsilon,
\]
then~\eqref{GN} holds for $R_0=R/2$.
\end{cor}
\begin{proof} According to~\cite{PetersenWei-97,PetersenWei-01},~\eqref{doubling} holds up to radius $R$ for some $\epsilon>0$, while  the heat kernel upper bound follows from~\cite{Rose-17a, DaiWeiZhang-18} by choosing $\epsilon$ possibly smaller.
\end{proof}

\begin{proof}[Proof of Theorem~\ref{main2}] The proof is adapted from~\cite{ChoulliKayserOuhabaz-15}.
If $\Omega\subset M$ is $(r,H,K)$-regular, then it satisfies in particular the rolling $r$-ball condition for some $r\leq r_0$, and $r_0$ depends on $K$ and $H$ only. According to~\cite{Olive-19}, there exists a $C>0$ depending on $R,p,n$ such that
\begin{align}\label{omegadoubling}
\Vol_\Omega(B(x,t))\leq C\left(\frac ts\right)^n\Vol_\Omega(B(x,s)), \quad x\in\Omega, 0<s\leq t\leq \diam \Omega\leq R.
\end{align}
Moreover, according to Theorem~\ref{main1}, $(r,H,K)$-regularity implies that there is an extension operator $E_\Omega$ with norm bounded in terms of $H,K,r$. Note that by construction, $\Vert E_\Omega\Vert_{L^2(\Omega),L^2(M)}$ is bounded as well and does not depend on any curvature restrictions but on the rolling $r$-ball condition. Moreover, we have
\begin{equation*}
  \Vol_\Omega(B(x,t))\leq \Vol(B(x,t)),\quad t>0.
\end{equation*}
Abbreviate $A=\Vert E_\Omega\Vert_{L^2(\Omega),L^2(M)}$ and $B=\Vert E_\Omega\Vert_{H^1(\Omega),H^1(M)}$. By the local Gagliardo-Nirenberg inequality, Corollary~\ref{cor:GN} and the existence of a extension operator $E_\Omega$ from Theorem~\ref{main1}, for any $s\in (0,R/2]$, $q\in[2,\infty]$, $\frac{q-2}qn<2$, $f\in \cC^1(\overline\Omega)$, we have
\begin{align*}
\Vert \Vol_\Omega(B(x,s))^{1/2-1/q}f\Vert_{L^q(\Omega)}
&\leq \Vert v_s^{1/2-1/q} E_\Omega f\Vert_{L^q(\Omega)}\\
&\leq C(\Vert E_\Omega f\Vert_{L^2(M)}+s \Vert\nabla E_\Omega f\Vert_{L^2(M)})\\
&\leq C(A\Vert f\Vert_{L^2(\Omega)}+  B s (\Vert f\Vert_{L^2(\Omega)}+\Vert \nabla f\Vert_{L^2(\Omega)}))\\
&\leq C\max(A,B)((1+s)\Vert f\Vert_{L^2(\Omega)}+s\Vert\nabla f\Vert_{L^2(\Omega)})\\
&\leq C\max(A,B)(1+\diam(\Omega))(\Vert f\Vert_{L^2(\Omega)}+s\Vert\nabla f\Vert_{L^2(\Omega)})\\
&\leq C(n,p,r,K,H,R)(\Vert f\Vert_{L^2(\Omega)}+s\Vert \nabla f\Vert_{L^2(\Omega)}),
\end{align*}
i.e., the Gagliardo-Nirenberg inequality on $\Omega$ for all $s\in(0,R/2]$. For the case $s>R/2$, note that $\Vol_\Omega(B(x,s))=\Vol(\Omega)$ for $s\geq\diam \Omega$, hence the case $s>\diam(\Omega)$ reduces to the case $s=\diam(\Omega)$. Thus, the Gagliardo Nirenberg inequality holds for all $s>0$ on $\Omega$. To derive the desired Neumann heat kernel upper bound, we want to apply \cite[Theorem~1.1]{BoutayebCoulhonSikora-15} directly. More precisely, the latter theorem shows that global volume doubling and global Gagliardo-Nirenberg inequalities on $\Omega$ yield an all-time upper bound for the Neumann heat kernel. The only thing that is still needed to check is whether inside $\Omega$, volumes of different balls of the same radius are comparable, i.e., condition $(D_v')$ in the notation of the latter paper for $v=\Vol_\Omega$. If $s>0$ and $x,y\in\Omega$, $d(x,y)\leq s$, then $B(y,s)\subset B(x,2s)$, such that~\eqref{omegadoubling} implies 
\begin{align*}
\Vol_\Omega(B(y,s))\leq \Vol_\Omega(B(x,2s))\leq 2^n C \Vol_\Omega(B(x,s)).
\end{align*}
Hence, the theorem follows, and the constants appearing depend only on the dimension, the doubling constant and radius, and the heat kernel upper bound.
\end{proof}

\begin{proof}[Proof of Corollary~\ref{cor:kato}] According to~\cite{RoseWei-20}, there exists an explicit constant $\epsilon=\epsilon(n,r,H,K)>0$ such that if
\begin{align}\label{Katoexp}
\kappa_T(\rho_-)\leq \int_0^T\Vert H_t^\Omega \rho_-\Vert_\infty\drm t\leq \epsilon,
\end{align}
then all the conclusions of the corollary hold.  Here $(H_t^\Omega)_{t\geq 0}$ is the Neumann heat semigroup of $\Omega$. First, note that by the Dunford-Pettis theorem, Theorem~\ref{main2}, and volume doubling on $\Omega$, we have for all $t\leq R^2$
\[
\Vert H^\Omega_t\Vert_{1,\infty}=\sup_{x,y\in\Omega} h_t^\Omega(x,y)\leq \sup_{x,y\in\Omega}\frac{\bar C}{\Vol_\Omega(B(x,\sqrt t))^{1/2}\Vol_\Omega(B(y,\sqrt t))^{1/2}}\leq \frac{\tilde C}{\Vol(\Omega)}\left(\frac R{\sqrt t}\right)^n.
\]
Thus, since $\Vert H_t^\Omega\Vert_{\infty,\infty}\leq 1$ and by duality, the Riesz-Thorin interpolation theorem implies
\[
\Vert H^\Omega_t\Vert_{p,\infty}\leq C_R\Vol(\Omega)^{-1/p} t^{-n /2p}.
\]
Hence,
\[
\kappa_T(\rho_-)\leq \int_0^T \Vert H_t^\Omega\rho_-\Vert_\infty\drm t\leq \int_0^T \Vert H_t^\Omega\Vert_{p,\infty}\Vert\rho_-\Vert_{p,\Omega}\drm t \leq C_R\Vol(\Omega)^{-1/p}\Vert\rho_-\Vert_{p,\Omega}\int_0^T t^{-n /2p}\drm t,
\]
and the latter function is integrable provided $n /2p<1$, i.e., $p>n/2$. Thus, the result follows by taking $T=R^2$ and forcing the right-hand side to be smaller than $\epsilon$.
\end{proof}
\appendix
\section{Appendix: Proofs of Propositions~\ref{hessiancomparison} and ~\ref{prop:volumecomparison}}\label{appendix}
We derive a differential inequality for the metric tensor along a geodesic depending on upper and lower bounds of the sectional and mean curvature that we compare with the solution $\lambda$ of the differential equation

\begin{equation}\label{comparisonfunction}
\lambda'+\lambda^2=-k,\quad \lambda(0)=h,
\end{equation}
where $k,h\in\RR$ are a lower (resp.\ upper) bound for the sectional curvature along the geodesic and upper (resp.\ lower) bound on the principal curvatures of $N$.
By substituting $\lambda=\frac {\mu'}{\mu+C}$ for $C\in\RR$, this equation transforms into the solvable differential equation
\begin{equation}\label{substitution}
\mu''+k\mu=-kC, \quad \mu'(0)=h(\mu(0)+C).
\end{equation}
Choosing $\mu(0)=1$ and $C$ appropriately guarantees the existence of a $t_0>0$ and a unique solution $\mu_{k,H}$ that is positive on an interval $(0,t_0]$.

\begin{proof}[Proof of Proposition~\ref{hessiancomparison}]
We only show Equation~\eqref{lowerbound} by following the proof of~\cite[Theorem~27]{Petersen-06}. The upper bound for the metric can be proven similarly. Observe that the initial conditions for the Hessian of $s$ are given by
\[
\Hess s(0)=\II_{c(0)}.
\]
Fix $\theta\in N$ and define
\[
\lambda(s):=\lambda(s,\theta):=\max_{v\perp \partial_s}\frac{\Hess s(v,v)}{g(v,v)},
\]
that is Lipschitz and hence absolutely continuous. Let $v$ be a vector such that 
\[
\Hess s(v,v)=\lambda(s_0)g_s(v,v),
\]
where $s_0$ is a point where $\lambda$ is differentiable, and extend $v$ to a parallel field $V$. The function $\phi(s):=\Hess s(V,V)$ satisfies $\phi(s)\leq \lambda(s)$ and $\phi(s_0)=\lambda(s_0)$. This yields, since $V$ is parallel,
\begin{align*}
\lambda'(s_0)+\lambda^2(s_0)&=\partial_s\Hess s(v,v)+\Hess s^2(v,v)=\nabla_{\partial_s}\Hess s(v,v)+\Hess s^2(v,v)\\
&=-g(R(v,\partial_s)\partial_s,v)\leq -k. 
\end{align*}
Moreover, $\lambda(0)=H_+$.
Thus, by Riccati comparison, the Hessian satisfies
\[
\Hess s\leq \frac{\mu_{k,H_+}'(s)}{\mu_{k,H_+}(s)}g_s.
\]
In general, we have 
\[
\partial_s g_s=2\Hess s,
\]
yielding
\begin{align}\label{proof:inequ}
\partial_sg_s\leq 2 \frac{\mu_{k,H_+}'(s)}{\mu_{k,H_+}(s)}g_s.
\end{align}
To get the desired estimate for the metric, we compare this differential inequality with the conformal variation
\[
h_s=\mu_{k,H_+}^2(s)g_0, \quad \mu_{k,H_+}(0)=1, s\in (0,r_0).
\]

This variation satisfies
\[
\partial_s h_s=2\frac{\mu_{k,H_+}'(s)}{\mu_{k,H_+}(s)}h_s, \quad h_0=g_0.
\]
Comparing this equality with Equation~\eqref{proof:inequ} yields the claim.
\end{proof}

To get Proposition~\ref{prop:volumecomparison}, we adopt the comparison argument from~\cite{Petersen-06} to our situation. In general for the decomposition~\eqref{metricdecomposition} of our metric, we have 
\[
\partial_s \dvol= \Delta s \dvol
\]
where $\Delta s=\Tr \II$ denotes the mean curvature of the distance hypersurface. If we decompose the volume element into 
\[
\dvol = \lambda (s,\theta) \drm s \dvol_s,
\]
we see that the equation above for fixed $\theta\in N$ reduces to
\[
\lambda'(s)=\Delta s \  \lambda(s),
\]
where $\Delta s(0)=H:=\Tr \II(0)$. We want to compare this metric with the conformal variation 
\[
h_s=\mu_{k,H}^2(s)g_0,
\]
on $N$. Note that we have 
\[
\dvol_s^{k,H}=\mu_{k,H}^{n-1}(s)\dvol_0,
\]
and 
\[
\partial_s(\dvol^{k,H}_s)=\partial_s \mu_{k,H}^{n-1}\dvol_0=(n-1)\frac{\mu_{k,H}'(s)}{\mu_{k,H}(s)} \dvol_s^{k,H},\quad \frac{\mu_{k,H}'(0)}{\mu_{k,H}(0)}=HC.
\]
The mean curvature can be controlled by the following result by Eschenburg.
\begin{lem}[{\cite[Theorem~4.1]{Eschenburg-87}}] For $M,N,c,r_0, H$ as above, we have 
\begin{align}\label{meancomparison}
\Delta s\leq \mu_{k,H}(s), \quad s\in[0,r_0),  m_{k,H}(0)=\Tr \II(0).
\end{align}
For $s<0$, the opposite inequality holds.
\end{lem}

\bibliographystyle{alpha}

\providecommand{\bysame}{\leavevmode\hbox to3em{\hrulefill}\thinspace}
\providecommand{\MR}{\relax\ifhmode\unskip\space\fi MR }
\providecommand{\MRhref}[2]{%
  \href{http://www.ams.org/mathscinet-getitem?mr=#1}{#2}
}
\providecommand{\href}[2]{#2}

\end{document}